\documentclass[11pt,leqno]{article}

\usepackage{amssymb,amsmath,amsthm,mysects,url,rotating}
 \usepackage{graphicx,epsfig}
 \usepackage{xcolor,graphicx}
\newcommand{\n}{\noindent}

\newcommand{\vp}{\varepsilon}
\newcommand{\bb}[1]{\mathbb{#1}}
\newcommand{\cl}[1]{\mathcal{#1}}

\theoremstyle{plain}
\newtheorem{thm}{Theorem}[section]
\newtheorem{lem}[thm]{Lemma}

\theoremstyle{definition}

\theoremstyle{remark}
\newtheorem{rem}[thm]{Remark}

\numberwithin{equation}{section}

\def\R{\bb R}
\def\RR{\bb R}
\def\C{\bb C}

\def\E{\bb E}
\def\F{\bb F}
\def\P{\bb P}

\def\NN{\bb N}

\def\RR{\bb R}

\begin{document}

\title{Examples of  non exact $1$-subexponential $C^*$-algebras }

\author{by\\
Gilles  Pisier\\
Texas A\&M University\\
College Station, TX 77843, U. S. A.\\
and\\
Universit\'e Paris VI\\
IMJ, Equipe d'Analyse Fonctionnelle, Case 186,\\ 75252
Paris Cedex 05, France}

 \maketitle

\begin{abstract} 
This is a supplement to our previous paper on the arxiv \cite{Pp}.
We show that there is a non-exact $C^*$-algebra that is
$1$-subexponential, and we give several other complements to the
results of that paper.  Our example can be described very simply using random matrices:
Let $\{X_j^{(m)}\mid j=1,2,\cdots\}$ be an i.i.d. sequence of random $m\times m$-matrices 
distributed according to the Gaussian Unitary Ensemble (GUE).
For each $j$ let  $u_j(\omega)$ be the block direct sum defined by
$$u_j(\omega)= \oplus_{m\ge 1} X_j^{(m)}(\omega)\in  \oplus_{m\ge 1} M_m.$$
Then for almost every $\omega$ the $C^*$-algebra generated by
$\{u_j(\omega) \mid j=1,2,\cdots\}$ is  $1$-subexponential but is not exact.\\
The GUE is a matrix model for the semi-circular distribution. 
We can also use instead the analogous circular model. 
\end{abstract}

\n Consider the direct sum $B=\oplus_{m\ge 1} M_{m}$. By definition, for any $x=\oplus_{m\ge 1} x(m)\in B$
we have $\|x\|=\sup_{m\ge 1}\| x(m)\|.$
   We equip $M_{m}$ with its normalized trace $\tau_m$. 
           
          Let $u_j=\oplus_m u_j(m)$ be elements of $B$. Let  $\cl A$ be the  unital $C^*$-algebra generated by $u_1,u_2,\cdots ,u_n$. For simplicity we set $u_0=1$.
          Let $\cl C$ be a  unital $C^*$-algebra that we assume generated by $c_1,c_2,\cdots $ 
          and equipped with a faithful
          tracial state $\tau$. We again set $c_0=1$.
          
          We say (following \cite{M}) that $\{u_j(m)\mid 0\le j\le n\}$ tends strongly to $\{c_j\mid 0\le j\le n\}$
          when $m\to \infty$ if  it tends weakly 
          (meaning ``in moments" relative to $\tau_m$ and $\tau$) and moreover
          $  \|P(u_i(m))\| \to  \|P(c_i)\|$ for any (non-commutative) polynomial $P$.
          This implies that
          for any $n+1$-tuple of such polynomials $P_0,P_1,\cdots,P_n$, for any $k$ and any $a_j\in M_k$ we have
          \begin{equation}\label{eqq0--} \lim_{m\to \infty} \|\sum\nolimits_0^n a_j \otimes P_j (u_i(m)) \|= \|\sum\nolimits_0^n a_j \otimes P_j(c_i) \|.\end{equation}
         In particular we have
          \begin{equation}\label{eqq0-} \lim_{m\to \infty} \|\sum\nolimits_0^n a_j \otimes u_j(m) \|= \|\sum\nolimits_0^n a_j \otimes c_j \|.\end{equation}
          Let $I_0\subset B$ denote the ideal of sequences $(x_m)
\in  B$ that tend to zero in norm (usually denoted by $c_0(\{M_{N_m}\})$.
          Let $Q:\ B\to B/I_0$ be the quotient map.
          It is easy to check that for any polynomial $P$ we have
          $\|Q(P(u_j))\|=\|P(c_j)\|$.
          So that, if we set $I=I_0\cap \cl A$,  we  have a natural identification
          $$\cl A/I= \cl C.$$

         Let $P_d$ denote the linear space of all polynomials of degree $\le d$
         in the non commutative variables $ (X_1,\cdots,X_n,X^*_1,\cdots,X^*_n)$.
         We will need to consider the space $M_k \otimes P_d$.  
         It will be convenient to systematically use the following notational convention:
         $$\forall 1\le j\le n\quad X_{n+j}=X_j^*.$$
         A typical element
         of $M_k \otimes P_d$ can then be viewed as a polynomial $P=\sum a_{J}\otimes X^J$ with coefficients
         in $M_k$. Here the index $J$ runs over the   disjoint union 
         of the sets $\{1,\cdots,2n\}^i$ with $1\le i\le d$. We also
         add symbolically  the value $J=0$ to the index set and we set $X^0$ equal to the unit.

         We denote by $P(u(m))\in M_k\otimes M_m$ (resp. $P(c)\in M_k\otimes \cl C$)
         the result of substituting $\{u_j(m)\}$ (resp. $\{c_j\}$)    in place of $\{X_j\}$.
         It follows from the strong convergence of
         $\{u_j\mid 0\le j\le n\}$  to $\{c_j\mid 0\le j\le n\}$ that
         for any $d$ and any $P\in M_k \otimes P_d$ we have
         $$\|P(u(m))\|\to \|P(c)\|.$$
         With a similar convention we will write e.g.
         $P(c)=\sum a_{J}\otimes c^J$.\\
  In particular this implies (actually this already follows from weak convergence)
          \begin{equation}\label{ee1} 
       \forall k\  \forall d \ \forall P\in  M_k \otimes P_d\quad   \|P(c)\|\le \liminf_{m\to \infty}\|P (u(m))  \|.
          \end{equation} 
                     \begin{rem}\label{r1} Let us  write $P$ as a sum of monomials $P=\sum a_{J}\otimes X^J$
         as above. We will assume that the  operators $\{c^J\}$ are linearly independent.
           From this assumption
          follows that there is a constant $c_2(n,d)$ such that
         $$\sum_{J} \|a_{J}\|\le c_2(n,d) \|P(c)\|.$$
         Indeed, since the span of the $c^J$'s is finite dimensional,  the linear form that takes $P$ to its $c^J$-coefficient  
         is continuous, and its norm (that depends obviously only on $(n,d)$)
         is the same as its c.b. norm. Of course this depends also on the distribution of the family  $\{c_j\}$ but we view this as fixed from now on.
         \end{rem}
         
    We will consider the following assumption: 
   
\begin{equation}\label{eea2}
 \sum_1^n \tau(|c_j|^2) >  \|   \sum_1^n u_j \otimes \bar c_j \|_{{\cl A}\otimes_{\min}  \bar {\cl C}} .     \end{equation} 
                
                   \n {\bf Notation.} Let $\alpha\subset \NN$ be a subset (usually infinite in the sequel).
                   We denote $$B(\alpha)= \oplus_{m\in \alpha} M_m.$$
                   $$u_j(\alpha)=\oplus_{m\in \alpha} u_j(m)\in B(\alpha).$$
                   We will denote by $A(\alpha)\subset B(\alpha)$ the unital  $C^*$-algebra   generated   by
            $\{ u_j(\alpha)\mid 1\le j\le n\}$. With this notation  $\cl A=A(\NN)$. \\
            We also set 
             $E_d(\alpha) = \{P(u(\alpha))\mid P\in P_d\}.$
          It will be convenient to set also $u^\alpha_j(m)=0$ whenever $m\not\in \alpha$.
          \def\a{\alpha}

          Fix a degree $d\ge 1$. Then for any real numbers $m\ge 1$  and $t\ge 1$ we define
          $$C_d(m,t)=\sup_{m'\ge m}\sup_{k\le t}\{ \|P(u(m'))\|\mid P\in M_k\otimes P_d,\ \|P(c)\|\le 1\}.$$
            \begin{thm} \label{t1} Assume that for any $d\ge 1$ there are $a>0$ and $D>0$
            such that $C_d(aN^D,N)\to 1$ when $N \to \infty$. Assume moreover that
            \eqref{ee1} holds. Then for any subset $\alpha\subset \NN$ the unital  $C^*$-algebra $A(\alpha)$ generated by
            $\{ u_j(\alpha)\mid 1\le j\le n\}$ is $1$-subexponential.
            Moreover, if we assume \eqref{eea2}  then it is not exact.
             \end{thm}

       \begin{proof}  For subexponentiality, we need to show that for any fixed $\vp>0$ and any finite dimensional subspace $E\subset A(\a)$ the growth of $N\mapsto K_E(N,1+\vp)$ is subexponential. Since the polynomials in $\{u_j(\a)\}$ are dense in $A(\alpha)$, by perturbation it suffices to check this
       for $E\subset E_d(\alpha) $. Thus we may as well assume $E=E_d(\alpha) $.\\
       Then we may choose $N_0$ large enough so that $C_d(aN^D,N)<1+\vp$ for all $N\ge N_0$.
       We claim that for all $N\ge N_0$ we have $K_E(N,(1+\vp)^2)\in O(N^{2D})$ when $N\to \infty$.
       To verify this, let $P\in M_N\otimes P_d$. 
       Then, recalling \eqref{ee1}, we have 
        \begin{equation}\label{ee2} \|P(c)\|\le \sup_{m\ge aN^D}
\|P(u(m))\|\le C_d(aN^D,N)\|P(c)\|.\end{equation}
Let $\a'= \a \cap [1,aN^D)$.
Let $T:\ E \to        B(\a')\oplus \cl C$ be the linear mapping
defined for all $P$ in $P_d$ by
$$T(P(u(\a))= P(u(\a'))\oplus P(c).$$
We may assume $\a$ infinite (otherwise the subexponentiality is trivial). Then \eqref{ee1} shows
that $\|T\|_{cb}\le 1$. Conversely, by \eqref{ee2} we have
$$\|(T^{-1})_N\|\le C_d(aN^D,N)<1+\vp.$$
Let $\hat E$ be the range of $T$.
 This shows that $d_N(E,\hat E)< 1+\vp.$ 
 We have $\hat E\subset \oplus_{k< aN^D} M_k \oplus \hat E'$
 where $\hat E'$ is a finite dimensional subspace of $\cl C$ (included in the span of polynomials of degree $d$). 
  Since $\cl C$ is exact, there is an integer $K$ such that $\hat E'$
  is completely $(1+\vp)$-isomorphic to a subspace of  $ M_K$, so that
   $\hat E$ is completely $(1+\vp)$-isomorphic to a subspace of  $ \oplus_{k< aN^D} M_k \oplus M_K$.
  Therefore we have for any $N\ge N_0$
  $$K_E(N,(1+\vp)^2)\le 1+2+\cdots+ [aN^D]+K$$
  and hence our claim follows, proving the $1$-subexponentiality.\\
  We now show that $A(\alpha)$ is not exact. Recall the notation
  $B(\alpha)=\oplus_{m\in \a} M_m$. 
  By Kirchberg's results (see e.g. \cite[p. 286]{P4}), if $A(\alpha)$  is exact then
  the inclusion map $V:\ A(\alpha)\to B(\alpha)$ satisfies the following:
  for any $C^*$-algebra $C$ the mapping
    $V\otimes Id_{C}:\ A(\alpha)\otimes_{\min} C \to B(\alpha)\otimes_{\max} C$ 
    is bounded (and is actually contractive). 
    Let $\cl U$ be any free ultrafilter on $\a$. 
    \def\u{{\cl U}}
    Let $M^\u$ denote the von Neumann algebra ultraproduct of $\{M_m\mid m\in \a\}$, with each 
    $M_m$ equipped with $\tau_m$.
    Recall that $M^\u$  is finite (cf. e.g. \cite[p. 211]{P4}).
    We may view $\cl C$ as embedded in $M^\u$. 
        Let $\cl M$ be the von Neumann algebra generated by $\cl C$.
        Note that $\cl M$ can also be identified (as von Neumann algebra)
        with  the von Neumann algebra generated by $\cl C$   
    when we view it as  embedded in $B(L_2(\tau))$.
   
 We have
    a quotient map $Q_1:\ B(\alpha)\to M^\u$ and
    a (completely contractive) conditional expectation $Q_2$  from  $M^\u$ to $\cl M$.
    Let $q:  \ A(\alpha)\to \cl M$ be the composition $q=Q_2Q_1 V$. 
    By the above, $q\otimes Id_{C}:\ A(\alpha)\otimes_{\min} C \to \cl M\otimes_{\max} C$ 
   must be  bounded (and   actually contractive). However, if we take
   $C=\bar {\cl C} $, this implies since $c_j=q(u_j)$
   $$   \|\sum_1^n c_j \otimes \bar c_j \|_{{\cl M}\otimes_{\max}  \bar {\cl C}}\le 
     \|   \sum_1^n u_j \otimes \bar c_j \|_{{A(\alpha)}\otimes_{\min}  \bar {\cl C}}\le
    \|   \sum_1^n u_j \otimes \bar c_j \|_{{\cl A}\otimes_{\min}  \bar {\cl C}} .$$
    But now  using the fact that left and right multplication acting on $L_2(\tau)$ are commuting representations on $\cl M$,
    we immediately find
    $$ \sum_1^n \tau(|c_j|^2) \le   \|\sum_1^n c_j \otimes \bar c_j \|_{{\cl M}\otimes_{\max}  \bar {\cl C}} $$
    and this contradicts \eqref{eea2}. This contradiction shows that $A(\alpha)$
    is not exact.
 \end{proof}
           
              \begin{rem}\label{r2}  
            Let  $Y^{(m)}$ denote a random $m\times m$-matrix with i.i.d. 
  complex Gaussian entries with mean zero and second moment equal to $m^{-1/2}$, and let
     $(Y_j^{(m)})$ be a sequence of i.i.d. copies of $Y^{(m)}$.
     We will use the matrix model formed by these matrices
              (sometimes called the ``Ginibre ensemble"), for which it
            is known (\cite{VDN}) that we have weak convergence to a free circular family $\{c_j\}$.
            Moreover, by \cite{HT3} we have also almost surely strong convergence
            of the random matrices to the free circular system. Actually, the inequalities
            from \cite{HT3} that we will crucially use are stated there mostly for the GUE ensemble, i.e. for 
            self adjoint Gaussian matrices with a semi-circular weak limit.
            These can be defined simply by setting
            $$X_j^{(m)}=\sqrt {2 }\Re(Y_j^{(m)}).$$
            Note  we also have an identity  in distribution $s_j= \sqrt {2 } \Re ( c_j)$.
            We call this the self-adjoint model.
            However, as explained in \cite{HT3} , it is easy to pass from one setting to the other by
            a simple ``$2 \times 2$-matrix trick".
            Since we prefer to work in the circular setting, we will now indicate this trick.

            When working in the self-adjoint model,  of course we consider only polynomials
            of degree $d$ in $ (X_1,\cdots,X_n)$. Fix $k$. Then
            the set  of polynomials of degree $\le d$ with coefficients in $M_k$ of the form $P(X_j^{(m)})$
            is included in the corresponding set of polynomials of degree $\le d$
            of the form $P(Y_j^{(m)})$. Conversely, any $P(Y_j^{(m)})$ can be viewed
            as a polynomial of degree $\le d$ in $ (X_1^{(m)},\cdots,X_{2n}^{(m)})$. Indeed, the real and imaginary
            parts of  $Y_j^{(m)}$ are independent copies of $X_j^{(m)}$. 
            This is clear when the coefficients are arbitrary in $M_k$. However, the results
            of \cite{HT3}  are stated for self-adjoint coefficients $a_J$  in  $ M_k$.
            Then the trick consists in replacing the general coefficients $a_J$
            by self-adjoint ones defined by
         $$\hat a_J =\left(\begin{matrix} 0 \ \ a_J\\
            a_J^* \ \  0\end{matrix} \right)\in M_{2k}.
         $$
         Let $\hat P= \sum \hat a_J \otimes X^J$. One then notes that $\|\hat P(s)\|=
         \|  P(s)\|$ and similarly $\|\hat P(X_j^{(m)})\|=
         \|  P(X_j^{(m)})\|$. Thus by simply passing from $k$ to $2k$
         we can deduce the strong convergence for general coefficients,
         as expressed in  \eqref{eqq0--} and \eqref{eqq0-} from   
         the   case of self-adjoint coefficients. 
            
               \end{rem}           
           The following Lemma is well known.
             
        \begin{lem} Let $F$ be any scalar valued random variable that is in $L_p$ for all $p<\infty$.
        Fix $a>0$.
        Assume that 
        $$\sup_{p \ge 1}  p^{-a}\|F\|_p \le \sigma.$$
        Then
        $$\forall t>0 \quad \P\{ |F|>t\} \le e \exp - (e\sigma)^{-1/a} t^{1/a} . $$
       
               \end{lem}
         \begin{proof} By Tchebyshev's inequality, for any $t>0$ we have
         $t^p  \P\{ |F|>t\}\le  (\sigma p^a)^p $, and hence
         $ \P\{ |F|>t\} \le  (t^{-1}\sigma {p^a})^p \le \exp -p\log (t /(\sigma {p^a}))$.
         Assuming $ t/(e\sigma )\ge 1$, we can choose $p= (  t/(e\sigma ) )^{1/a}$ and then we find
         $\P\{ |F|>t\} \le   \exp -(e \sigma)^{-1/a} t^{1/a}  $ and, a fortiori, the inequality holds. Now if  $ t/(e\sigma )< 1$, we have
         $\exp {-(e\sigma)^{-1/a} t^{1/a} }> e^{-1} $ and hence $e \exp{- (e\sigma)^{-1/a} t^{1/a} }>1$ so that the inequality 
         trivially holds.
   \end{proof}
   We will use concentration of measure in the following form:
           
             \begin{lem}\label{lem3} There is a constant $c_1(n,d)>0$   such that               for any $k$ and any
                $P\in M_k\otimes P_d$ with $\|P(c)\|\le 1$, we have
                $$\forall t>0\quad \P\{ |\|P(Y^{(m)})\|-\E\|P(Y^{(m)})\|| > t \} \le e\exp-(t^{2/d} m^{1/d}/ c_1(n,d)).$$
               
               \end{lem}
               \def\b{\beta}
         \begin{proof}    This  follows from
         a very general concentration inequality for Gaussian random vectors, that can be derived in various ways.  We choose the following for which we refer to \cite{P02}.
         Consider any sufficiently smooth function
         (meaning a.e. differentiable) $f:\ \R^n \to \R$ and let $\P$ denote the canonical Gaussian measure on
         $\R^n$. Assuming $f\in L_p(\P)$ we have 
         $$ \|f-\E f\|_p\le (\pi/2)\| Df(x).y\|_{L_p(\P(dx)\P(dy))}. $$
         Let $\gamma(p)$ denote the $L_p$-norm of a standard normal Gaussian variable
         (in particular $\gamma(p)=\|f\|_p $ for $f(x)=x_1$). 
         Recall that $\gamma(p)\in O(\sqrt{p} )$ when $p\to \infty$. Thus the last inequality implies that there is a constant $\b $ such that
         $$ \|f-\E f\|_p\le \b \sqrt{p} \| \|Df(x)\|_2 \|_{L_p(\P(dx)}, $$
         where $\|Df(x)\|_2$ denotes the Euclidean norm of the gradient of $f$ at $x$. 
         Clearly this remains true  for any $f$ on $\C^n$ (with the gradient computed on $\R^{2n}$).
         
         We will apply this
         to a function $f$ defined on $(\C^{m^2})^n$.
         We need to first clarify the notation. We identify $\C^{m^2}$ with $M_m$. Then we 
          define    $f$  on    $(\C^{m^2})^n$  by
         $$f(w_1,\cdots,w_n)   =\| g(w_1,\cdots,w_n)   \|$$
         with 
          $$g(w_1,\cdots,w_n)   =P(m^{-1/2}w_1,\cdots,m^{-1/2}w_n, m^{-1/2}w^*_1,\cdots,m^{-1/2}w^*_n ).$$ 
         Note that  for this choice of $f$ the derivative $D_z$ in any direction $z$ satisfies
         $D_zf \le \| D_zg\|$ and hence taking the sup over $z$ in the  Euclidean  unit sphere, we have pointwise
         $$\|Df\|_2 \le \sup_z \|D_zg\|.$$
         
         We now invoke Remark \ref{r1}.
Using the bound in that remark, we are 
         reduced  to the case when $P=Y^J$. Then  $D_zg$ is the sum
         of at most $d$ terms of the form $m^{-1/2} a z_i b$ 
         so that $\|m^{-1/2} a z_i b\|\le m^{-1/2}\|a\|\|z_i\|\|b\|$
         and hence since $\|z_i\|\le \|z_i\|_2 $,  $\|m^{-1/2} a z_i b\|\le m^{-1/2}\|a\| \|b\|$.
        Recollecting all the terms , this yields
        a pointwise estimate at the point $w \in M_m^{n}$
        $$\sup_z \|D_z g\| \le c_3(n,d) m^{-1/2} \sup\{ \|m^{-1/2}w_j\|\mid 1\le j\le n\}^{d-1}.$$
        Thus we obtain
      $$ \|f-\E f\|_p\le \b \sqrt{p} c_3(n,d) m^{-1/2} \| \sup_{1\le j\le n}  \|Y^{(m)}_j\| ^{d-1}\|_p,$$
      and a fortiori 
      $$ \|f-\E f\|_p\le \b\sqrt{p}  c_3(n,d) m^{-1/2} \| \sum_{1\le j\le n}  \|Y^{(m)}_j\| ^{d-1}\|_p\le \b\sqrt{p}  c_3(n,d) m^{-1/2} n\|  \|Y^{(m)}_1\| ^{d-1}\|_p,$$
         Now by general results on integrability
         of Gaussian vectors (see \cite[p. 134]{Led}),
         we know that
         there is an absolute  constant $c_5$ such that 
         $$\|\|Y^{(m)}_1\| ^{d-1}\|_p= \|Y^{(m)}_1\|^{d-1}_{L_{p(d-1)}(M_m)}       
            \le (c_5\sqrt{ p(d-1)} \E \|Y^{(m)}_1\| )^{d-1} $$
            and since we know that $\E \|Y^{(m)}_1\| \to 2$ when $m\to \infty$  
            it follows that $\|\|Y^{(m)}_1\| ^{d-1}\|_p\le (c_{10} \sqrt{ p(d-1)}  )^{d-1} $.
            Thus we obtain
            $$ \|f-\E f\|_p\le c_4(n,d) m^{-1/2} p^{d/2},$$
         and the conclusion follows from the preceding Lemma.
              \end{proof}

                 \begin{rem}\label{r3}  It will be convenient to record here an elementary consequence of
                 Lemma \ref{lem3}.
                 Let $F= \|P(Y^{(m)})\|$ and let  $t_m=  \E\| P(Y^{(m)})\|$, so that 
                 we know
                 $\forall t>0\quad \P\{ F > t+t_m \} \le \psi_m(t)$
                 with $$\psi_m(t)=
                  e\exp-(t^{2/d} m^{1/d}/ c_1(n,d)).$$
                 We have
                $$ \E\left((F/2-t_m) 1_{\{F/2>t_m\}}\right)=\int_{t_m}^\infty \P\{F/2>t\} dt\le  \int_{t_m}^\infty \P\{F>t+t_m\} dt\le \int_{t_m}^\infty  \psi_m(t)dt $$
                and hence
                \begin{equation}\label{e10} \E F 1_{\{F/2>t_m\}} \le 2t_m \P{\{F/2>t_m\}}+  2\int_{t_m}^\infty  \psi_m(t)dt  .\end{equation}

                          \end{rem}
              
           The next result is a consequence of the results of {Haagerup} and {Thorbj{\o}rnsen}
            \cite{HT3} and of them with Schultz \cite{HT4}. Let us first recall
        the result from \cite{HT3} that we crucially need.
        
        \def\CB{{\mathcal B}}
 
   \begin{thm}[\cite{HT3,HT4}]\label{ht} Let $\chi_d(k,m) $ denote the best constant 
            such that
            for any $P\in M_k\otimes P_d$
            we have
            $$\E \|P( Y_j^{(m)}) \|\le \chi_d(k,m) \| P(c) \|. $$
            Then for any $0<\delta<1/4 $
            $$\lim_{m\to \infty }  \chi_d([m^{\delta}],m)=1.$$
               \end{thm}
         \begin{proof}
        By homogeneity we may assume $\| P(c) \|\le1$.
        Then by Remark \ref{r1} we also have
        \begin{equation}\label{e12}\sum_{J} \|a_{J}\|\le c_2(n,d).\end{equation}
      Fix $\vp>0$ and $t>1+\vp$.
Consider a    function        $\varphi\in
C_c^\infty(\RR,\RR)$ with values in $[0,1]$
such that $\varphi =0 $ on $[-1,1]$ 
and $\varphi (x)=1 $ for all $x$ such that $1+\vp<|x|<t$
and $\varphi (x)=0$ for  $|x|>2t$.
 Let $P^{(m)}=P( X_j^{(m)})$ and $P^{(\infty)}=P( s_j)$.
 By Remark \ref{r2} we can reduce our estimate
 to the case of a polynomial in $(X_j^{(m)})$ 
  with self-adjoint  coefficients and  with $(s_j)$ in place of $(c_j)$.
  Thus we now assume $\| P(s) \|\le1$.
  Clearly we still have    a bound of the form \eqref{e12}.
  Then  by  \cite{HT4} 
  (and by very carefully tracking the dependence of the various constants in \cite{HT4})  we have for $m\ge c_{13}(n,d)$
\begin{equation}\label{hot}
\E\big\{(\tau_k\otimes\tau_m)\varphi(P^{(m)})\big\} =
(\tau_k\otimes\tau)\varphi(P^{(\infty}))+R_m(\varphi)
\end{equation}
where
\begin{equation}
\label{eq5-12}
|R_m(\varphi)| \le   k^3 m^{-2} c_9(n,d) c_\vp t^3
\end{equation}
where $c_\vp$ depends only on $\vp$.
 Note $\varphi(P^{(\infty)})=0$.  Therefore
 \begin{equation}\label{htt}
\E\big\{(\tau_k\otimes\tau_m)\varphi(P^{(m)})\big\} \le k^3 m^{-2} c_9(n,d) c_\vp t^3.
\end{equation}
Since $\|P^{m}\|\in (1+\vp,t) \Rightarrow (\tau_k\otimes\tau_m)\varphi(P^{m})\ge 1/(km)$
  by Tchebyshev's inequality we find
$$ \P\{ \|P^{(m)}\|\in (1+\vp,t) \} \le  (km)k^3 m^{-2} c_9(n,d) c_\vp t^3=k^4m^{-1} c_9(n,d) c_\vp t^3.$$
        Thus we obtain
        $$\E \|P^{(m)}\|\le 1+\vp + k^4 m^{-1} c_9(n,d) c_\vp t^4+ \E (\|P^{(m)}\| 1_{\{\|P^{(m)}\|> t\}}).  $$
We will now invoke \eqref{e10}:  choosing $t=2t_m=2\E \|P^{(m)}\|$ we find
 $$\E \|P^{(m)}\|\le 1+\vp + k^4 m^{-1} c_9(n,d) c_\vp t_m^4+   2t_m \psi_m(t_m) +  2\int_{t_m}^\infty  \psi_m(t) .  $$
 Now by \eqref{e12} and by H\"older
 we  have 
 $$t_m\le c_2(n,d) \sup_J \E\| {X^{(m)}}^J \|\le c_2(n,d) \sup_J \E(\|X_1^{(m)}\|^{|J|}) $$
 but  by a well known result essentially due to Geman \cite{G} (cf. e.g. \cite[Lemma 6.4]{S}),
 for any $d$ we have
 $$c_9(d)=\sup_m  \E(\|X_1^{(m)}\|^{d})<\infty.$$
 Therefore we have  $t_m\le c'_2(n,d)$.
 We may assume $t_m>1$ (otherwise there is nothing to prove) and
 hence  we have proved 
 $$\E \|P^{(m)}\|\le 1+\vp + k^4m^{-1} c'_9(n,d) c_\vp  +   2c'_2(n,d) \psi_m(1) +  2\int_{1}^\infty  \psi_m(t)dt . $$
Thus for any $\vp>0$ we conclude 
 $$\chi_d(k,m)\le 1+\vp + k^4m^{-1} c'_9(n,d) c_\vp  +   2c'_2(n,d) \psi_m(1) +  2\int_{1}^\infty  \psi_m(t)dt.$$
 From this estimate
 it follows clearly that for any $0<\delta<1/4 $
 $$\limsup_{m\to \infty}  \chi_d([m^{\delta}],m)\le 1+\vp  .$$
  
 \end{proof}

            \begin{lem}\label{lem4} Fix integers $d,k,m$. Let $\chi_d(k,m) $ denote the best constant 
            appearing in Theorem \ref{ht}.
            Then for any $\vp>0$ there are positive constants $c_7(n,d,\vp)$ and $c_8(n,d,\vp)$ such that
            if $k$ is the largest integer such that  $m\ge c_7(n,d,\vp) k^{2d}$ 
            the set
  $$\Omega_{d,\vp}(m)=
   \{   \forall P\in M_k\otimes P_d\quad   \|P(Y^{(m)}(\omega)) \|\le  (1+\vp) (\chi_d(k,m)+\vp) \| P(c) \|  \}$$
            satisfies
            $$\P(\Omega_{d,\vp}(m)^c)\le e\exp\left(  -m^{1/d}/ c_8(n,d,\vp)  \right).$$
            
       \end{lem}
         \begin{proof} 
         For any $P\in M_k\otimes P_d$ with $\|P(c)\|\le 1$,
         we have by Lemma
          \ref{lem3}  for any $t>0$
          $$\P\{ \|P(Y^{(m)}) \|>t+\chi_d(k,m)\}\le   e\exp-(t^{2/d} m^{1/d}/ c_1(n,d)).$$
          Let $\cl N$ be a $\delta$-net in the unit ball
          of the space $P_d$ equipped with the norm $P\mapsto \|P(c)\|$.
          Since $\dim(M_k\otimes P_d)=c_6(n,d) k^2$ for some $c_6(n,d)$, it is known that we can find such a net with
          $$|\cl N|\le (1+2/\delta)^{c_6(n,d) k^2}.$$
          Let $\Omega_1= \{\forall a\in \cl N,  \|P(Y^{(m)}) \|>t+\chi_d(k,m)\}$.
          Clearly 
          $$\P(\Omega_1)\le |\cl N|  e\exp-(t^{2/d} m^{1/d}/ c_1(n,d))
           \le  e\exp\left (2 c_6(n,d)\delta^{-1}    k^2 -t^{2/d} m^{1/d}/ c_1(n,d)  \right) . $$
Thus if we choose $m$ so that (roughly )
$t^{2/d} m^{1/d}/ c_1(n,d)=4 c_6(n,d)\delta^{-1}    k^2$ we find an estimate of the form
$$\P(\Omega_1)\le  e\exp\left (  -t^{2/d} m^{1/d}/ 2c_1(n,d)  \right) . $$
Note that on the complement of $\Omega_1$ we have
$$\forall P\in \cl N \quad  \|P(Y^{(m)}) \|\le t+\chi_d(k,m) .$$
By a well known result  (see e.g. \cite[p. 49-50]{P-v}) we can pass from the set $\cl N$ to the whole unit ball at the cost
of a factor close to 1, namely we have on the complement of $\Omega_1$
$$\forall P\in M_k\otimes P_d\quad   \|P(Y^{(m)}) \|\le (1-\delta)^{-1} ( t+\chi_d(k,m) )\|P(c)\|.$$
Thus if we set $t= \vp$ and $\delta\approx \vp/2$, we obtain that
if $m=c_7(n,d,\vp) k^{2d}$
we have a set $\Omega_1'=\Omega_1^c$
with 
$$\P(\Omega_1')>1-  e\exp\left (  -\vp^{2/d} m^{1/d}/ 2c_1(n,d)  \right) ,$$
 such that for any $\omega\in\Omega_1'$  we have 
$$\forall P\in M_k\otimes P_d\quad   \|P(Y^{(m)}(\omega)) \|\le (1+\vp) (\chi_d(k,m)+\vp)\|P(c)\|.$$
        \end{proof}

         \begin{thm}\label{t2} For each $j$ let  $u_j(\a)(\omega)$ be the block direct sum defined by
$$u_j(\a)(\omega)= \oplus_{m\in \a} Y_j^{(m)}(\omega)\in  \oplus_{m\in \a} M_m.$$
Let  $\a\subset \NN$ be any infinite subset. Then  for almost every $\omega$ the $C^*$-algebra $A(\a)(\omega)$ generated by
$\{u_j(\a)(\omega) \mid j=1,2,\cdots\}$ is  $1$-subexponential but is not exact.\\
Moreover, these results remain valid
in the self-adjoint setting, if we replace $u_j(\a)(\omega)$ by
 $$\hat u_j(\a)(\omega)= \oplus_{m\in \a} X_j^{(m)}(\omega)\in  \oplus_{m\in \a} M_m.$$

               \end{thm}
         \begin{proof} We will give the proof for $\a=\NN$. The proof   for a general subset is
         identical. By Lemma \ref{lem4} for any degree        
         $d$ and $\vp>0$ we have 
         $$\sum_m \P(\Omega_{d,\vp}(m)^c)<\infty. $$
         Therefore the set $V_{d,\vp}=\liminf_{m\to \infty} \Omega_{d,\vp}(m)$
         has probability 1. Furthermore  (since we may use a sequence of $\vp$'s tending to zero) we have
         $$\P( \cap_{d,\vp} V_{d,\vp})=1.$$
         Now if we choose $\omega$ in $\cap_{d\ge 1,\vp>0} V_{d,\vp}$ the operators
        $u_j(\a)(\omega)$ satisfy the assumptions 
        of Theorem \ref{t1}, and hence $A(\a)(\omega)$ is $1$-subexponential.\\
        Note that $$   \sum_1^n \tau(|c_j|^2) =n.$$
Since $\|u_j\|=\sup_m \|u_j(m)\|$ and $\lim_{m\to \infty}\|u_j(m)(\omega)\|=2$ a.s.
        we know that $\sup_m \|u_j(m)\|<\infty$ a.s.. Therefore, by concentration
        (or by the integrability of the norm of Gaussian  random vectors, see \cite{Led})
        $$  \E (\|u_j\|^2) =\E (\sup_m \|u_j(m)\|^2)<\infty,$$
        and since the $u_j$'s have the same distribution  $\E (\|u_j\|^2)=\E (\|u_1\|^2)$.
        By Fatou's lemma $$\E \liminf_{n\to \infty} n^{-1} \sum_1^n \|u_j\|^2\le  \liminf_{n\to \infty} \E n^{-1} \sum_1^n \|u_j\|^2=\E( \|u_1\|^2)<\infty$$
         and hence
     there is a measurable set $\Omega_0\subset \Omega$ with $\P(\Omega_0)=1$
        such that
        $$\forall \omega\in \Omega_0\quad \liminf_{n\to \infty} n^{-1} \sum_1^n \|u_j(\omega)\|^2 <\infty.$$ 
        Therefore if we choose $\omega$ in the intersection
        of $ \cap_{d,\vp} V_{d,\vp}\cap \Omega_0$
        (which has probability 1)
        we find almost surely 
         $$\|   \sum_1^n u_j (\omega)\otimes \bar c_j \|_{{\cl A}\otimes_{\min}  \bar {\cl C}}\le
        2\max\{ \|\sum u_j u_j^* \|^{1/2}, \|\sum u_j^* u_j \|^{1/2}\}
        \le   2( \sum_1^n \|u_j(\omega)\|^2)^{1/2} \in O   (\sqrt{n} )$$
         so that \eqref{eea2} is satisfied and $A(\a)(\omega)$ is not exact.\\
         Lastly, since $\{\hat u_j(\a)(\omega)\mid j\in \a\}$ has  the same   distribution as $\{\sqrt{2}\Re u_j(\a)(\omega)\mid j\in \a\}$ the
         random $C^*$-algebra they generate has ``the same distribution" as  $A(\a)(\omega)$, whence the last assertion. 
                \end{proof}
                
            \begin{rem} It seems clear that our results remain valid if we replace $(Y_j^{(m)})$ 
            by an i.i.d. sequence of uniformly distributed $m\times m$ unitary matrices,
           but,
            at the time of this writing, we have not yet  been able to extract the suitable
            estimates (as in Theorem \ref{ht}) from the proof of Collins and Male \cite{CM} that strong convergence holds in this case.
            However, while this would simplify our example, by eliminating the need for estimates of $\|u_j\|$,
            it apparently would not significantly change the picture.
     \end{rem}
        \bigskip

        \n\textbf{Acknowledgment.}  I am  very grateful to Mikael de la Salle for useful remarks.

  \end{document}